\documentclass[12pt,reqno,
]{amsart}

\usepackage[arrow,matrix,curve]{xy}

\usepackage[dvips]{graphicx} 

\usepackage{amssymb, latexsym, amsmath, amscd, array, hyperref,
setspace    
}

\newcommand{\hr} {{{}^{\mathfrak{h}}\hspace*{-0.4pt}\R}}

\newtheorem{theorem}{Theorem}[section]

\theoremstyle{definition}

\numberwithin{equation}{section}

\newcommand\N {{\mathbb N}}

\newcommand\R {{\mathbb R}} 

\newcommand\Q {{\mathbb Q}} 

\newcommand\astr {{}^{\ast}\hspace*{-.7pt}{\mathbb R}}

\newcommand\astf{{}^{\ast}\hspace{-3pt}f}

\newcommand\astn{{{}^{\ast}\N}}

\newcommand\st{\textbf{st}}

\newcommand\Los{{\L}o{\'s}}

\author[V. Kanovei]{Vladimir Kanovei} \address{V. Kanovei, IPPI,
Moscow, and MIIT, Moscow, Russia}\email{kanovei@googlemail.com}

\author[K. Katz]{Karin U. Katz}\address{K. Katz, Department of
Mathematics, Bar Ilan University, Ramat Gan 52900
Israel}\email{katzmik@math.biu.ac.il}

\author[M. Katz]{Mikhail G. Katz}\address{M. Katz, Department of
Mathematics, Bar Ilan University, Ramat Gan 52900
Israel}\email{katzmik@macs.biu.ac.il}

\author[T. Mormann]{Thomas Mormann}\address{T. Mormann, Department of
Logic and Philosophy of Science, University of the Basque Country
UPV/EHU, 20080 Donostia San Sebastian, Spain}
\email{ylxmomot@sf.ehu.es}

\begin{document}

\thispagestyle{empty}

\title[What makes a theory of infinitesimals useful?]  {What makes a
theory of infinitesimals useful?  A view by Klein and Fraenkel}



\begin{abstract}
Felix Klein and Abraham Fraenkel each formulated a criterion for a
theory of infinitesimals to be successful, in terms of the feasibility
of implementation of the Mean Value Theorem.  We explore the evolution
of the idea over the past century, and the role of Abraham Robinson's
framework therein.
\end{abstract}

\maketitle


\section{Introduction}
\label{s1}

Historians often take for granted a historical continuity between the
calculus and analysis as practiced by the 17--19th century authors, on
the one hand, and the arithmetic foundation for classical analysis as
developed starting with the work of Cantor, Dedekind, and Weierstrass
around 1870, on the other.

We extend this continuity view by exploiting the Mean Value Theorem
(MVT) as a case study to argue that Abraham Robinson's framework for
analysis with infinitesimals constituted a continuous extension of the
procedures of the historical infinitesimal calculus.  Moreover,
Robinson's framework provided specific answers to traditional
preoccupations, as expressed by Klein and Fraenkel, as to the
applicability of rigorous infinitesimals in calculus and analysis.

This paper is meant as a modest contribution to the prehistory of
Robinson's framework for infinitesimal analysis.  To comment briefly
on a broader picture, in a separate article by Bair et al.\;\cite{17a}
we address the concerns of those scholars who feel that insofar as
Robinson's framework relies on the resources of a logical framework
that bears little resemblance to the frameworks that gave rise to the
early theories of infinitesimals, Robinson's framework has little
bearing on the latter.  Such a view suffers from at least two
misconceptions.  First, a hyperreal extension results from an
ultrapower construction exploiting nothing more than the resources of
a serious undergraduate algebra course, namely the existence of a
maximal ideal (see Section~\ref{s5}).  Furthermore, the issue of the
\emph{ontological} justification of infinitesimals in a set-theoretic
framework has to be distinguished carefully from the issue of the
\emph{procedures} of the early calculus which arguably find better
proxies in modern infinitesimal theories than in a Weierstrassian
framework; see further in B\l aszczyk et al.\;\cite{17d}.  For an
analysis of Klein's role in modern mathematics see Bair et
al.\;\cite{18a}.  For an overview of recent developments in the
history of infinitesimal analysis see Bascelli et al.\;\cite{14a}.

\section{Felix Klein}

In 1908, Felix Klein formulated a criterion of what it would take for
a theory of infinitesimals to be successful.  Namely, one must be able
to prove an MVT for arbitrary intervals (including infinitesimal
ones).  Writes Klein: ``there was lacking a method for estimating
\ldots the increment of the function in the finite interval.  This was
supplied by the \emph{mean value theorem}; and it was Cauchy's great
service to have recognized its fundamental importance and to have made
it the starting point accordingly of differential calculus''
\cite[p.\;213]{Kl08}.  A few pages later, Klein continues:
\begin{quote}
The question naturally arises whether \ldots{} it would be possible to
\emph{modify} the traditional foundations of infinitesimal calculus,
so as to include actually \emph{infinitely small} quantities in a way
that would satisfy modern demands as to rigor; in other words, to
construct a non-Archime\-dean system.  The first and chief problem of
this analysis would be to prove the mean-value theorem
\[
f(x+h)-f(x)=h \cdot f'(x+\vartheta h)
\]
[where $0\leq\vartheta\leq1$] from the assumed axioms.  I will not say
that progress in this direction is impossible, but it is true that
none of the investigators have achieved anything positive.
\cite[p.~219]{Kl08} (emphasis added)
\end{quote}
See also Kanovei et al.\;\cite[Section\;6.1]{13c}.  Klein's sentiment
that the axioms of the traditional foundations need to be modified in
order to accommodate a true infinitesimal calculus were right on
target.  Thus, Dedekind completeness needs to be relaxed; see
Section~\ref{s52}.

The MVT was still considered a research topic in Felix Klein's
lifetime.  Thus, in 1884 a controversy opposed Giuseppe Peano and
Louis-Philippe Gilbert concerning the validity of a proof of MVT given
by Camille Jordan; see Luciano \cite{Lu07}, Mawhin \cite{Ma11},
Besenyei \cite{Be12}, Smory\'nski \cite{Sm17} for details.

\section{Abraham Fraenkel}
\label{s3}

Robinson noted in his book that in 1928, Abraham Fraenkel formulated a
criterion similar to Klein's, in terms of the MVT.\; Robinson first
mentions the philosopher Paul Natorp of the Marburg school: ``during
the period under consideration attempts were still being made to
define or justify the use of infinitesimals in Analysis (e.g. Geissler
[1904], Natorp [1923])'' \cite [p.\;278] {Ro66}.  Robinson goes on to
reproduce a lengthy comment from Abraham Fraenkel's 1928 book
\cite[pp.\;116--117]{Fran} in German.  We provide a translation of
Fraenkel's comment:
\begin{quote}
\ldots With respect to this test the infinitesimal is a complete
failure.  The various kinds of infinitesimals that have been taken
into account so far and sometimes have been meticulously argued for,
have contributed nothing to cope with even the simplest and most basic
problems of the calculus. For instance, for [1] a proof of the mean
value theorem or for [2] the definition of the definite
integral. \ldots{} There is no reason to expect that this will change
in the future.''  (Fraenkel as quoted in Robinson \cite[p.~279]{Ro66};
translation ours; numerals [1] and [2] added)
\end{quote}
Thus Fraenkel formulates a pair of requirements: [1] the MVT and [2]
definition of the definite integral.  Fraenkel then offers the
following glimmer of hope:
\begin{quote}
Certainly, it would be thinkable (although for good reasons rather
improbable and, at the present state of science, situated at an
\emph{unreachable distance} [in the future]) that a second Cantor
would give an impeccable arithmetical foundation of new infinitely
small number that would turn out to be mathematically useful, offering
perhaps an easy access to infinitesimal calculus. 
(ibid., emphasis added)
\end{quote}
Note that Fraenkel places such progress at unreachable distance in the
future.  

This is perhaps understandable if one realizes that
Cantor--Dedekind--Weierstrass foundations, formalized in the
Zermelo--Fraenkel (the same Fraenkel) set-theoretic foundations, were
still thought at the time to be a primary point of reference for
mathematics (see Section~\ref{s1}).  Concludes Fraenkel: 
\begin{quote}
But as long this is not the case, it is not allowed to draw a parallel
between the certainly interesting numbers of Veronese and other
infinitely small numbers on the one hand, and Cantor's numbers, on the
other. Rather, one has to maintain the position that one cannot speak
of the mathematical and therefore logical existence of the infinitely
small in the same or similar manner as one can speak of the infinitely
large.%
\footnote{The infinities Fraenkel has in mind here are Cantorian
infinities.}
(ibid.)
\end{quote}
An even more pessimistic version of Fraenkel's comment appeared a
quarter-century later in his 1953 book \emph{Abstract Set Theory},
with MVT replaced by Rolle's theorem \cite[p.\;165]{Fr53}.

\section{Modern infinitesimals}

Fraenkel's 1953 assessment of ``unreachable distance''
notwithstanding, only two years later Jerzy \Los{} in \cite{Lo55}
(combined with the earlier work by Edwin Hewitt \cite{He48} in 1948)
established the basic framework satisfying the Klein--Fraenkel
requirements, as Abraham Robinson realized in 1961; see \cite{Ro61}.
The third, 1966 edition of Fraenkel's \emph{Abstract Set Theory} makes
note of these developments:
\begin{quote}
Recently an \emph{unexpected} use of infinitely small magnitudes, in
particular a method of basing analysis (calculus) on infinitesimals,
has become possible and important by means of a non-archimedean,
non-standard, proper extension of the field of the real numbers.  For
this surprising development the reader is referred to the literature.
\cite[p.\;125]{Fr66} (emphasis added)
\end{quote}
Fraenkel's use of the adjective \emph{unexpected} is worth commenting
on at least briefly.  Surely part of the surprise is a foundational
challenge posed by modern infinitesimal theories.  Such theories
called into question the assumption that the
Cantor--Dedekind--Weierstrass foundations are an inevitable
\emph{primary} point of reference, and opened the field to other
possibilities, such as the IST enrichment of ZFC developed by Edward
Nelson; for further discussion see Katz--Kutateladze \cite{15c} and
Fletcher et al.\;\cite{17f}.

This comment of Fraenkel's is followed by a footnote citing Robinson,
Laugwitz, and Luxemburg.  Fraenkel's appreciation of Robinson's theory
is on record:
\begin{quote}
my former student Abraham Robinson had succeeded in saving the honour
of infinitesimals - although in quite a different way than Cohen and
his school had imagined.  \cite{Fra67} (cf.\;\cite[p.\;85]{Fr16})
\end{quote}
Here Fraenkel is referring to Hermann Cohen (1842--1918), whose
fascination with infinitesimals elicited fierce criticism by both
Georg Cantor and Bertrand Russell.  For an analysis of Russell's
critique see Katz--Sherry \cite[Section\;11.1]{13f}.  For more details
on Cohen, Natorp, and Marburg neo-Kantianism, see Mormann--Katz
\cite{13h}.

\section{A criterion}
\label{s5}

Both Klein and Fraenkel formulated a criterion for the usefulness of a
theory of infinitesimals in terms of being able to prove a mean value
theorem.  Such a Klein--Fraenkel criterion is satisfied by the
framework developed by Hewitt, \Los, Robinson, and others.  Indeed,
the MVT
\[
(\forall x\in\R)(\forall h\in\R) (\exists \vartheta\in\R)
\big(f(x+h)-f(x)=h \cdot g(x+\vartheta h) \big)
\]
where $g(x)=f'(x)$ and $\vartheta\in[0,1]$, holds also for the natural
extension~$\astf$ of every real smooth function~$f$ on an arbitrary
hyperreal interval, by the \emph{Transfer Principle}; see
Section~\ref{s51}.  Thus we obtain the formula
\[
(\forall x\in\astr)(\forall h\in\astr) (\exists \vartheta\in\astr)
\big(\astf(x+h)-\astf(x)=h \cdot {}^\ast\hspace{-3pt}g(x+\vartheta h)
\big),
\]
valid in particular for infinitesimal $h$.  

\subsection{Transfer}
\label{s51}

The \emph{Transfer Principle} is a type of theorem that, depending on
the context, asserts that rules, laws or procedures valid for a
certain number system, still apply (i.e., are ``transfered'') to an
extended number system.  In this sense it is a formalisation of the
Leibnizian \emph{Law of Continuity}; such a connection is explored in
Katz--Sherry\;\cite{13f}.

Thus, the familiar extension~$\Q\hookrightarrow\R$ preserves the
property of being an ordered field.  To give a negative example, the
extension~$\R\hookrightarrow\R\cup\{\pm\infty\}$ of the real numbers
to the so-called \emph{extended reals} does not preserve the field
properties.  The hyperreal extension~$\R\hookrightarrow\astr$ (see
Section~\ref{s52}) preserves \emph{all} first-order properties.  The
result in essence goes back to \Los{} \cite{Lo55}.  For example, the
identity~$\sin^2 x+\cos^2 x=1$ remains valid for all hyperreal~$x$,
including infinitesimal and infinite inputs~$x\in\astr$.  Another
example of a transferable property is the property that $\text{for all
positive }x,y, \text{ if } x<y \text{ then } \frac{1}{y}<\frac{1}{x}$.
The Transfer Principle applies to formulas like that characterizing
the continuity of a function~$f\colon \R \to \R$ at a point~$c \in
\R$:
\[
(\forall\varepsilon > 0)(\exists\delta > 0)(\forall
x)\big[|x-c|<\delta \;\Rightarrow\; |f(x)-f(c)|<\varepsilon\big];
\]
namely, formulas that quantify over \emph{elements} of the field.  An
element~$u\in\astr$ is called \emph{finite} if~$-r<u<r$ for a
suitable~$r\in\R$.  Let~$\hr\subseteq\astr$ be the subring consisting
of finite elements of~$\astr$.  There exists a
function~$\st\colon\hr\to\R$ called \emph{the standard part}
(sometimes referred to as the \emph{shadow}) that rounds off each
finite hyperreal~$u$ to its nearest real number~$u_0\in\R$, so
that~$u_0=\st(u)$ and~$u\approx u_0$, where~$a\approx b$ is the
relation of infinite proximity (i.e., $a-b$ is infinitesimal).

\subsection{Extension}
\label{s52}

The hyperreal extension $\R\hookrightarrow\astr$ is the only modern
theory of infinitesimals that satisfies the Klein--Fraenkel criterion.
Here~$\astr$ can be obtained as the quotient of the ring of sequences
$\R^\N$ by a suitable maximal ideal.  The fact that it satisfies the
criterion is due to the transfer principle.  In this sense, the
transfer principle can be said to be a ``powerful new principle of
reasoning''.  Note that $\astr$ is not Dedekind-complete.

One could object that the classical form of the MVT is not a key
result in modern analysis.  Thus, in Lars H\"ormander's theory of
partial differential operators \cite[p.~12--13]{Ho}, a key role is
played by various multivariate generalisations of the following Taylor
(integral) remainder formula:
\begin{equation}
\label{71}
f(b)=f(a)+(b-a)f'(a)+\int_a^b (b-x)f''(x) dx.
\end{equation}
Denoting by~$\mathcal{D}$ the differentiation operator and by
$\mathcal{I}=\mathcal{I}(f,a,b)$ the definite integration operator, we
can state~\eqref{71} in the following more detailed form for a
function~$f$:
\begin{equation}
\label{72}
\begin{aligned}
(\forall & a\in\R) (\forall b\in\R) \\ & f(b)
=f(a)+(b-a)(\mathcal{D}f)(a)+ \mathcal{I} \left(
(b-x)(\mathcal{D}^2_{\phantom{I}} f),a,b \right)
\end{aligned}
\end{equation}
Applying the transfer principle to the elementary formula~\eqref{72},
we obtain
\begin{equation}
\label{73}
\begin{aligned}
(\forall & a\in\astr) (\forall b\in\astr) \\ &
\astf(b)=\astf(a)+(b-a)({}^*\mathcal{D}\;\astf)(a)+ {}^*\mathcal{I}
\left( (b-x)({}^*\mathcal{D}^2_{\phantom{I}}\; \astf),a,b \right)
\end{aligned}
\end{equation}
for the natural hyperreal extension~$\astf$ of~$f$.  The
formula~\eqref{73} is valid on every hyperreal interval of~$\astr$.
Multivariate generalisations of~\eqref{71} can be handled similarly.

\subsection{Mean Value Theorem}

We have focused on the MVT (and its generalisations) because,
historically speaking, it was emphasized by Klein and Fraenkel.  The
transfer principle applies far more broadly, as can be readily guessed
from the above.  The mean value theorem is immediate from Rolle's
theorem, which in turn follows from the extreme value theorem.  For
the sake of completeness we include a proof of the extreme value
theorem exploiting infinitesimals; see Robinson \cite[p.\;70,
Theorem\;3.4.13]{Ro66}.

\begin{theorem}
A continuous function~$f$ on~$[0,1]\subseteq\R$ has a maximum.
\end{theorem}

\begin{proof}
The idea is to exploit a partition into infinitesimal subintervals,
pick a partition point $x_{i_0}$ where the value of the function is
maximal, and take the \emph{shadow} (see below) of $x_{i_0}$ to obtain
the maximum.

In more detail, choose infinite hypernatural
number~$H\in\astn\setminus\N$.  The real interval~$[0, 1]$ has a
natural hyperreal extension ${}^\ast\hspace{-1.5pt}[0,1] =
\{x\in\astr\colon 0\leq x\leq1\}$.  Consider its partition into~$H$
subintervals of equal infinitesimal length~$\tfrac{1}{H}$, with
partition points $x_i = \tfrac{i}{H}, \quad i=0,\ldots,H$.  The
function~$f$ has a natural extension~$\astf$ defined on the hyperreals
between~$0$ and~$1$.  Among finitely many points, one can always pick
a maximal value: $(\forall n\in \N) \; (\exists i_0\leq n) \; (
\forall i\leq n ) \left( f(x_{i_0}) \geq f(x_i) \right)$.  By transfer
we obtain
\begin{equation}
\label{41}
(\forall n \in \astn) \; ( \exists i_0\leq n ) \; ( \forall i\leq n )
\; \left( \astf(x_{i_0}) \geq \astf(x_i) \right),
\end{equation}
where~$\astn$ is the collection of hypernatural numbers.
Applying~\eqref{41} to~$n=H\in \astn \setminus\N$, we see that there
is a hypernatural~$i_0$ such that~$0\leq i_0\leq H$ and
\begin{equation}
\label{42}
(\forall i \in \astn) \big[i\leq H \Longrightarrow \astf(x_{i_0})\geq
\astf(x_i)\big].
\end{equation}
Consider the real point $c=\st(x_{i_0})$ where \st{} is the
\emph{standard part function}; see Section~\ref{s51}.
Then~$c\in[0,1]$.  By continuity of~$f$ at $c\in\R$, we
have~$\astf(x_{i_0})\approx \astf(c)=f(c)$, and therefore
$\st\left(\astf(x_{i_0})\right)= \astf\left(\st(x_{i_0})\right)=f(c)$.
An arbitrary real point~$x$ lies in an appropriate sub-interval of the
partition, namely~$x\in[x_i,x_{i+1}]$, so that $\st(x_i)=x$, or
$x_i\approx x$.  Applying the function \st{} to the inequality in
formula~\eqref{42}, we
obtain~$\st(\astf(x_{i_0}))\geq\st(\astf(x_i))$.  Hence~$f(c) \geq
f(x)$, for all real~$x$, proving~$c$ to be a maximum of~$f$ (and by
transfer, of~$\astf$ as well).
\end{proof}
The partition into infinitesimal subintervals (used in the proof of
the extreme value theorem) similarly enables one to define the
definite integral as the shadow of an infinite Riemann sum, fulfilling
Fraenkel's \emph{second} requirement, as well; see Section~\ref{s3}.

The difficulty of the Klein--Fraenkel challenge was that it required a
change in foundational thinking, as we illustrated.

\section{Acknowledgments}

V. Kanovei was partially supported by the RFBR grant no.\;17-01-00705.
M.\;Katz was partially supported by the Israel Science Foundation
grant no.\;1517/12.

\end{document}